\documentclass{amsart}

\usepackage{amscd}
\usepackage{xcolor}
\usepackage[pdftitle={...},
 pdfauthor={Alcides Buss, Siegfried Echterhoff,  Rufus Willett},
 pdfsubject={Mathematics}
]{hyperref}
\usepackage[lite]{amsrefs}

\usepackage{verbatim}

\usepackage{amsfonts}

\usepackage{amsthm}

\usepackage{amssymb}

\usepackage{amsmath}

\usepackage{mathabx}

\usepackage{enumerate}

\usepackage[all]{xy}

\usepackage{graphicx}

\usepackage{hyperref}

\DeclareMathOperator{\Aut}{Aut}

\newcommand{\C}{\mathbb{C}}

\newcommand{\K}{\mathcal{K}}

\newcommand{\Manoa}{M\=anoa}

\newcommand{\Hawaii}{Hawai\kern.05em`\kern.05em\relax i}

\newcommand{\alg}{\text{alg}}

\newcommand{\inj}{{\operatorname{inj}}}
\newcommand{\Ind}{{\operatorname{Ind}}}

\newcommand{\B}{\mathcal{B}}
\newcommand{\M}{\mathcal{M}}

\newcommand*{\into}{\hookrightarrow}
\newcommand*{\onto}{\twoheadrightarrow}
\newcommand*{\red}{r}
\renewcommand*{\max}{\mathrm{max}}

\newcommand*{\cont}{C}
\newcommand*{\contz}{\cont_0}
\newcommand*{\contub}{\cont_{ub}}
\newcommand*{\dual}[1]{\widehat{#1}}

\newcommand*{\sbe}{\subseteq} 

\newcommand*{\cstar}{\texorpdfstring{$C^*$\nobreakdash-\hspace{0pt}}{*-}}

\theoremstyle{plain}
\newtheorem{theorem}{Theorem}[section]
\newtheorem{lemma}[theorem]{Lemma}
\newtheorem{corollary}[theorem]{Corollary}
\newtheorem{proposition}[theorem]{Proposition}

\newtheorem{definition-theorem}[theorem]{Definition / Theorem}

\newtheorem*{conjecture*}{Conjecture}
\newtheorem*{theorem*}{Theorem}

\theoremstyle{definition}
\newtheorem{definition}[theorem]{Definition}
\newtheorem{example}[theorem]{Example}

\theoremstyle{remark}
\newtheorem{remark}[theorem]{Remark}

\newtheorem*{example*}{Example}  
\newtheorem*{remark*}{Remark}

\begin{document}
\title{The maximal injective crossed product}
\author{Alcides Buss}
\email{alcides@mtm.ufsc.br}
\address{Departamento de Matem\'atica\\
 Universidade Federal de Santa Catarina\\
 88.040-900 Florian\'opolis-SC\\
 Brazil}

\author{Siegfried Echterhoff}
\email{echters@uni-muenster.de}
\address{Mathematisches Institut\\
 Westf\"alische Wilhelms-Universit\"at M\"un\-ster\\
 Einsteinstr.\ 62\\
 48149 M\"unster\\
 Germany}

\author{Rufus Willett}
\email{rufus@math.hawaii.edu}
\address{Mathematics Department\\
 University of \Hawaii~at \Manoa\\
Keller 401A \\
2565 McCarthy Mall \\
 Honolulu\\
 HI 96822\\
USA}

\begin{abstract}
A crossed product functor is said to be \emph{injective} if it takes injective morphisms to injective morphisms.  In this paper we show that every locally compact group $G$ admits a maximal injective crossed product $A\mapsto A\rtimes_{\inj}G$.  Moreover, we give an explicit construction of this functor that depends only on the maximal crossed product and the existence of $G$-injective $C^*$-algebras; this is a sort of a `dual' result to the construction of the minimal exact crossed product functor, the latter having been studied for its relationship to the Baum-Connes conjecture.  It turns out that $\rtimes_\inj$ has  interesting connections to exactness, the local lifting property, amenable traces, and the weak expectation property.
\end{abstract}


\maketitle

\section{Introduction}
Given a fixed locally compact group $G$, an (exotic) crossed product functor $\rtimes_\mu$ for $G$ is a functor $A\mapsto A\rtimes_\mu G$ from the 
category of $G$-$C^*$-algebras into the category of $C^*$-algebras in which  $A\rtimes_\mu G$ is a $C^*$-completion of the algebraic crossed product $A\rtimes_\alg G=C_c(G,A)$ with respect to a $C^*$-norm $\|\cdot \|_\mu$ which satisfies
$$\| \cdot \|_r\leq \|\cdot\|_\mu\leq \|\cdot\|_\max,$$
where $\|\cdot\|_r$ and $\|\cdot\|_\max$ denote the norms of the reduced and maximal crossed products, respectively. 
Thus we see that for any exotic crossed product the identity map on $C_c(G,A)$ induces surjections
$$A\rtimes_{\max}G\onto A\rtimes_\mu G\onto A\rtimes_r G.$$
Recently, the study of exotic crossed products has become a focus of research not only because of interesting connections to the 
Baum-Connes conjecture, as revealed in \cites{Baum:2013kx, Buss:2014aa, Buss:2015ty} but also because of the fact that exotic crossed products and group
algebras provide interesting new examples of $C^*$-algebras attached to locally compact groups and their actions (e.g., see \cites{Brown:2011fk, 
Kaliszewski:2012fr, Wiersma:2015tx, Wiersma:2016kr, Ruan:2016kl}).

A crossed product functor $A\mapsto A\rtimes_\mu G$ is said to be \emph{injective} if it takes injective morphisms to injective morphisms.  Our goal in this paper is to show that there is always a maximal injective crossed product $\rtimes_{\inj}$; this is a sort of `dual' result to the existence of a minimal exact crossed product functor that has been studied for its relationship to the Baum-Connes conjecture (we refer to \cite{Buss:2018pw} for the most recent results on this functor).  There are no applications given here to Baum-Connes, but it turns out that $\rtimes_\inj$ has some interesting connections to exactness, the local lifting property, amenable traces, and the weak expectation property, as well as $G$-injective algebras generally; our goal is to elucidate these.  We hope these show that $\rtimes_\inj$ is a natural object.

After this introduction we start with a preliminary section giving a self-contained introduction to $G$-injective $C^*$-algebras which have been studied  first by Hamana in \cite{Hamana:1985aa}. The main fact we need  in this paper is the observation that every $G$-algebra embeds equivariantly into a $G$-injective one 
(see Corollary \ref{g inj exist} below). The construction of the maximal injective crossed product $A\rtimes_\inj G$ is given in Section \ref{basic sec}.
Our results show that it can be described as the completion of the algebraic crossed product $A\rtimes_{\alg} G$ inside the maximal 
crossed product $B\rtimes_\max G$ if $A\into B$ is any $G$-equivariant embedding of $A$ into a $G$-injective $C^*$-algebra $B$. 
We show that this construction gives indeed an injective crossed product functor which is maximal among all injective crossed product functors for $G$.
It follows from this that $A\rtimes_\inj G=A\rtimes_\max G$ for every $G$-injective algebra $A$. In fact, 
in  Section \ref{basic sec} we show that a similar statement holds for all $G$-algebras which satisfy a $G$-equivariant version of Lance's weak expectation property
(see Proposition \ref{prop-GWEP}).

In Section \ref{sec LLP} we study some connections between properties of the injective crossed product $\rtimes_\inj$, exactness of $G$,  and 
 the local lifting property (LLP) of $C^*_\max(G)$. Using a recent characterisation of exact locally compact groups as those groups  which admit 
amenable actions on compact spaces due Brodzki, Cave and Li in \cite{Brodzki:2015kb}, it is fairly easy to show that the injective functor 
coincides with the reduced crossed product functor  for all exact $G$. This implies the interesting observation that for exact $G$ the reduced crossed product functor is the {\em  only injective crossed product functor} for $G$, although (if $G$ is not amenable) 
there are often uncountably many crossed product functors with other good properties (e.g. see \cite{Buss:2015ty} for a detailed discussion).
To make sure that we do not talk about  reduced crossed products only,  we show 
 that for a certain class of 
 non-exact groups as constructed by Osajda  (\cite{Osajda:2014ys}), we do have $\rtimes_{\inj}\neq \rtimes_r$. 
Another class of groups for which $\rtimes_{\inj}= \rtimes_r$ is given by those discrete groups $G$ for which the maximal group algebra $C_\max^*(G)$ 
has the local lifting property (LLP) (see Definition \ref{ex def}).  Together with the above result on Osajda's groups, this shows that for all these groups 
the group algebra $C_\max^*(G)$ does not have the local lifting property. Indeed, our results suggest that LLP for $C_\max^*(G)$ 
would possibly imply exactness, while we observe that the converse direction fails, since the only other known examples 
of groups with $C_\max^*(G)$ not having LLP due to Thom in \cite{Thom:2010aa} turn out to be exact.

In Section \ref{sec amenable} we study the group algebra $C_\inj^*(G)=\C\rtimes_\inj G$ associated to $\rtimes_\inj$.
Extending a well-known result for $C_r^*(G)$, we show that $C_\inj^*(G)$ admits an amenable trace if and only if $G$ is amenable,
which hints at some similarity between $C_\inj^*(G)$ and $C_r^*(G)$.
Since the trivial representation $1_G: C^*_\max(G)\to \C$ is always an amenable trace, this
 directly implies that $C_\inj^*(G)=C_\max^*(G)$ if and only if $G$ is amenable. 
There is a certain similarity between the defining properties of an amenable trace and Lance's weak expectation property (WEP)
 for a $C^*$-algebra $B$, and for
discrete groups we can show that if $A\rtimes_\inj G$ has the WEP, then  $A\rtimes_\inj G=A\rtimes_\max G$.
If $A=\C$ this implies that $C_\inj^*(G)$ has the WEP if and only if  $G$ is amenable, which gives another variant of a well known result for 
the reduced group algebra $C_r^*(G)$ (see \cite{Brown:2008qy}*{Proposition 3.6.9}).
In particular, if $G$ is discrete and exact, then our result shows that the WEP for $A\rtimes_rG$ implies that 
$A\rtimes_rG\cong A\rtimes_\max G$, which indicates that in general
 the  WEP for $A\rtimes_rG$ should be related to some kind of amenability for the action of $G$ on $A$. 

In Section \ref{sec subgroups} we show that the injective crossed product behaves quite naturally with respect to closed subgroups.
It turns out that the injective crossed product functor for a locally compact group $G$ always `restricts' to the injective functor for $M$,
for every closed subgroup $M$ of $G$. Moreover, if $M$ is an open supgroup of $G$  and $A$ is a $G$-algebra, then $A\rtimes_\inj M$ always 
embeds faithfully into $A\rtimes_\inj G$, a fact well known for the maximal and the reduced crossed products.
We close this paper in Section \ref{sec questions} with a short discussion of various open questions regarding the maximal injective 
crossed product functor.

The authors would like to thank Matthew Wiersma for several interesting discussions on the subject.  

Most of the work on this paper was carried out during a visit of the first and third authors to the second author at the Westf\"{a}lische Wilhelms-Universit\"{a}t M\"{u}nster; these authors would like to thank the second author, and that institution, for their hospitality.  

The authors were supported by Deutsche Forschungsgemeinschaft (SFB 878, Groups, Geometry \& Actions), by CNPq/CAPES -- Humboldt, and by the US NSF (DMS 1401126 and DMS 1564281)

\section{Preliminaries on $G$-injectivity}\label{g inj sec}

In this section, we give some background on $G$-injectivity for a locally compact group $G$.  This is presumably well-known to at least some experts, but we provide details for the reader's convenience and as we could not find exactly what we wanted in the literature.

\begin{definition}\label{g inj}
An equivariant ccp map $\phi:A\to B$ is \emph{$G$-injective} if for any equivariant injective $*$-homomorphism $\iota:A\to C$ the dashed arrow below 
$$
\xymatrix{ C \ar@{-->}[dr]^-{\widetilde{\phi}} &  \\ A \ar[u]^-\iota \ar[r]^-\phi & B } 
$$
can be filled in with an equivariant ccp map.

A $G$-$C^*$-algebra $B$ is \emph{$G$-injective} if any equivariant ccp map $A\to B$ is $G$-injective.
\end{definition}

For the next result, recall that a $C^*$-algebra $B$ is \emph{injective} if it is $G$-injective in the above sense for $G$ the trivial group.  The most important example is $\mathcal{B}(H)$ for any Hilbert space $H$: this is a consequence of Arveson's extension theorem (see \cite{Brown:2008qy}*{Theorem~1.6.1}).  The following result is essentially the same as \cite{Hamana:1985aa}*{Lemma 2.2}, but as we work in a slightly different context, we give a proof for the reader's convenience.

\begin{proposition}\label{inj exist}
Let $B$ be an injective $C^*$-algebra, and let $C_{ub}(G,B)$ denote the $C^*$-algebra of uniformly continuous (for the left-invariant uniform structure on $G$) bounded functions from $G$ to $B$, equipped with the $G$-action defined by 
$$
\gamma_g(f)(h):=f(g^{-1}h).
$$ 
Then $C_{ub}(G,B)$ is $G$-injective.
\end{proposition}

\begin{proof}
Let $\phi:A\to C_{ub}(G,B)$ be an equivariant ccp map, and let $\iota:A\to C$ be an equivariant embedding.  Let $\delta_e:C_{ub}(G,B)\to B$ be defined by $f\mapsto f(e)$, and let $\psi:A\to B$ be defined by $\psi=\delta_e\circ \phi$, which is ccp.  Injectivity of $B$ gives a ccp extension $\widetilde{\psi}:C\to B$ of $\psi$ to $C$.  Let now $\alpha$ denote the action of $G$ on $C$, and define  
$$
\widetilde{\phi}:C\to C_{ub}(G,B), \quad \widetilde{\phi}(c)(g):=\widetilde{\psi}(\alpha_{g^{-1}}(c)).
$$
This function is equivariant, and has the property that $\widetilde{\phi}(c)(g)$ is positive or contractive for each $g\in G$ whenever $a$ has these properties.  Using the identification $M_n(C_{ub}(G,B))=C_{ub}(G,M_n(B))$, this gives that $\widetilde{\phi}$ is ccp; we leave the straightforward check that it extends $\phi$ to the reader.
\end{proof}

\begin{remark}\label{action no matter}
If $B$ is equipped with a non-trivial $G$-action $\beta$, we may also consider $C_{ub}(G,B)$ equipped with the `diagonal' type action 
$$
\gamma_g(f)(h):=\beta_gf(g^{-1}h).
$$
The resulting $G$-$C^*$-algebra is equivariantly isomorphic to the one from Proposition \ref{inj exist}, so is also $G$-injective whenever $B$ is (non-equivariantly) injective.
\end{remark}

\begin{corollary}\label{g inj emb exist}
For any $G$-$C^*$-algebra $A$ there exists a $G$-injective $C^*$-algebra $B$ and an equivariant embedding of $A$ into $B$.
\end{corollary}

\begin{proof}
Choose a faithful representation $\pi:A\to \mathcal{B}(H)$. Since $\mathcal{B}(H)$ is injective, $B=C_{ub}(G,\mathcal{B}(H))$ is $G$-injective by the result of Proposition \ref{inj exist}.  Then if $\alpha$ denotes the action of $G$ on $A$, the map
$$
\widetilde{\pi}:A\to B,\quad \widetilde{\pi}(a)(g):=\pi(\alpha_{g^{-1}}(a))
$$
is an equivariant embedding.
\end{proof}

Lemma \ref{inj exist} also allows us to give another example of $G$-injective maps; this will be useful later.

\begin{corollary}\label{g inj exist}
Let $(A,\alpha)$ be a $G$-$C^*$-algebra, and fix a (not necessarily equivariant, or faithful) non-degenerate $*$-representation $\pi:A\to \mathcal{B}(H)$.  Let $\widetilde{\pi}:A\to \mathcal{B}(L^2(G,H))$ denote the representation used in the definition of the left regular representation: explicitly, 
$$
(\widetilde{\pi}(a)\xi)(g):=\pi(\alpha_{g^{-1}}(a))\xi(g);
$$
note that $\widetilde{\pi}$ is equivariant for the inner action on $\mathcal{B}(L^2(G,H))$ induced by the amplification of the regular representation.  We let $\mathcal{B}(L^2(G,H))_c$ denote the $G$-continuous part of this algebra.  

Then the $*$-homomorphism $\widetilde{\pi}:A\to \mathcal{B}(L^2(G,H))_c$ is $G$-injective.  In particular, if $\pi$ is injective $\widetilde{\pi}$ will be a $G$-injective embedding.
\end{corollary}

\begin{proof}
The map $\widetilde{\pi}:A\to \mathcal{B}(L^2(G,H))_c$ factors through $C_{ub}(G,\mathcal{B}(H))$ as in Proposition \ref{inj exist}, when the latter is included by multiplication operators in $\mathcal{B}(L^2(G,H))_c$ in the natural way.  It is straightforward to see that a map that factors through a $G$-injective algebra is $G$-injective.
\end{proof}

\section{The maximal injective crossed product}\label{basic sec}

In this section, $G$ denotes a general locally compact group.

\begin{definition}\label{g inj cp}
Let $A$ be a $G$-$C^*$-algebra.  For each equivariant injective $*$-homomorphism $\iota:A\to B$, define the \emph{$\iota$-norm} on $C_c(G,A)$ by 
$$
\|a\|_\iota:=\|(\iota\rtimes G)(a)\|_{B\rtimes_{\max} G},
$$ 
and let $A\rtimes_\iota G$ denote the corresponding completion.  Define the \emph{injective norm} on $C_c(G,A)$ by 
$$
\|a\|_\inj:=\inf\{\|a\|_{\iota}\mid \iota:A\to B \text{ an equivariant injection}\}
$$
and the \emph{injective crossed product} $A\rtimes_\inj G$ to be the corresponding completion.
\end{definition}

Note that $\|\cdot\|_\inj$ dominates the reduced norm (by injectivity of the latter), so it is a norm on $C_c(G,A)$, not just a seminorm.

In order to study $\rtimes_\inj$, we will make heavy use of the material on $G$-injective maps in the previous section.  

\begin{lemma}\label{realize inj}
Let $\pi:A\to B$ be a $G$-injective embedding.  Then $A\rtimes_{\inj}G$ identifies with the closure of $C_c(G,A)$ under its natural embedding in $B\rtimes_{\max} G$.
\end{lemma}

\begin{proof}
Let $\phi:A\to C$ be any $G$-equivariant embedding.  We need to show that $\|a\|_\pi\leq \|a\|_\phi$ for any $a\in C_c(G,A)$.  As $\pi$ is $G$-injective, there exists a ccp equivariant map $\psi:C\to B$ making the diagram 
$$
\xymatrix{ C \ar[dr]^-\psi & \\ A \ar[u]^-\phi \ar[r]^-\pi & B }
$$
commute.  The maximal crossed product is functorial for ccp maps, as follows for example from \cite{Buss:2014aa}*{Theorem 4.9}.  Taking maximal crossed products we thus get a commutative diagram
$$
\xymatrix{ C\rtimes_{\max}G \ar[dr]^-{\psi\rtimes G} & \\ A\rtimes_{\max} G \ar[u]^-{\phi\rtimes G} \ar[r]^-{\pi\rtimes G} & B\rtimes_{\max}G },
$$
where $\phi\rtimes G$ and $\pi\rtimes G$ are $*$-homomorphisms, and $\psi\rtimes G$ is ccp.  It follows that for any $a\in C_c(G,A)$,
$$
\|a\|_\pi=\|(\pi\rtimes G)(a)\|=\|(\psi\rtimes G)\circ (\phi\rtimes G)(a)\|\leq \|(\phi\rtimes G)(a)\|=\|a\|_\phi
$$
as required.
\end{proof}

From the above lemma, we get the following immediate corollary.

\begin{corollary}\label{cor max}
Suppose that $A$ is $G$-injective. Then $A\rtimes_{\inj}G=A\rtimes_\max G$. \qed
\end{corollary}

\begin{proposition}\label{inj func}
The injective crossed product defines a crossed product functor.
\end{proposition}

\begin{proof}
Let $\phi:A\to C$ be an arbitrary $*$-homomorphism.  Let $\pi_A:A\to B_A$ and $\pi_C:C\to B_C$ be $G$-injective $*$-homomorphisms as in Corollary \ref{g inj emb exist}.  Now, in the diagram
$$
\xymatrix{ B_A \ar@{-->}[r] & B_C \\ \ar[u]^-{\pi_A} A \ar[r]^-\phi & C \ar[u]^-{\pi_C} }
$$
the definition of injectivity for $B_C$ applied to the inclusion $A\to B_A$ and to the composition $\pi_C\circ \phi:A\to B_C$ allows us to fill in the dashed arrow with an equivariant ccp map, say $\widetilde{\phi}$.  Applying Lemma \ref{realize inj} and using functoriality of the ccp maps, we thus get a diagram 
$$
\xymatrix{ B_A\rtimes_{\max} G \ar[r]^{\widetilde{\phi} \rtimes G} & B_C\rtimes_{\max} G \\ A\rtimes_\inj G \ar[u]^-{\pi_A\rtimes G}  & C\rtimes_\inj G \ar[u]^{\pi_C\rtimes G}}
$$
where the vertical maps are injections.  Identifying $A\rtimes_\inj G$ and $C\rtimes_\inj G$ with their images under the vertical maps, the restriction of $\widetilde{\phi}\rtimes G$ to $A\rtimes_\inj G$ is the map required by functoriality.
\end{proof}

\begin{proposition}\label{inj inj}
The functor $\rtimes_\inj$ is the maximal injective crossed product functor.
\end{proposition}

\begin{proof}
We must first show that $\rtimes_\inj$ is injective.  Let then $\phi:A\to C$ be an injective equivariant $*$-homomorphism.  Let $\pi_A:A\to B_A$ and $\pi_C:C\to B_C$ be $G$-injective embeddings as in Corollary \ref{g inj emb exist}, and consider the diagram
$$
\xymatrix{ B_A & B_C \ar@{-->}[l] \\ A\ar[u]^{\pi_A} \ar[r]_-\phi & C \ar[u]_-{\pi_C} }.
$$
The composition $\pi_C\circ \phi:A\to B_C$ is injective, so as the inclusion $A\to B_A$ is $G$-injective the dashed arrow can be filled in with an equivariant ccp map, say $\psi$.  Taking crossed products gives a diagram
$$
\xymatrix{ B_A\rtimes_{\max} G & B_C\rtimes_{\max} G \ar@{-->}[l]_-{\psi\rtimes G} \\ A\rtimes_\inj G\ar[u]^{\pi_A\rtimes G} \ar[r]_-{\phi\rtimes G} & C\rtimes_\inj G \ar[u]_-{\pi_C\rtimes G} }
$$
where the vertical maps are injective by Lemma \ref{realize inj}.  The composition $(\psi\rtimes G)\circ (\pi_C\rtimes G)\circ (\phi\rtimes G)$ is thus injective as it agrees with the left hand vertical map; hence $\phi\rtimes G$ is injective as required.

To see that $\rtimes_\inj$ is the maximal injective crossed product, let $\rtimes_\mu$ be any other injective crossed product, and let $A\to B_A$ be a $G$-injective embedding.  Then we look at the diagram:
$$
\xymatrix{ A\rtimes_\inj G \ar[r] & B_A\rtimes_{\max} G \ar[d] \\ A\rtimes_\mu G \ar[r] & B_A\rtimes_\mu G }.
$$
The composition $A\rtimes_\inj G \to B_A\rtimes_\mu G$ has image isomorphic to $A\rtimes_\mu G$ as $\rtimes_\mu$ is injective, and is the identity on $C_c(G,A)$; this yields a homomorphism $A\rtimes_\inj G\onto A\rtimes_\mu G$ extending the identity on $C_c(G,A)$, completing the proof.
\end{proof}

\begin{remark}\label{rem-corr} Recall from \cite{Buss:2014aa} that  a crossed product 
 functor $\rtimes_\mu$ for $G$ is called a {\em correspondence functor} if it is functorial
 for $G$-equivariant correspondences in the sense that if $E$ is a $G$-equivariant correspondence from $A$ to $B$, then there is a canonical  construction 
 of a crossed product correspondence $E\rtimes_\mu G$  from $A\rtimes_\mu G$ to $B\rtimes_\mu G$.  It has been shown in \cites{Buss:2014aa, Buss:2015ty} 
 that correspondence functors enjoy many nice properties. For example, they admit dual coactions and a descent in Kasparov's $G$-equivariant bivariant $K$-theory (see \cite{Buss:2014aa}*{Sections 5 and 6}).
Moreover, \cite{Buss:2014aa}*{Theorem 4.9} shows (among other things) that for a given crossed product functor $\rtimes_\mu$
the following are equivalent:
\begin{enumerate}
\item $\rtimes_\mu$ is a correspondence functor.
\item $\rtimes_\mu$ is injective on $G$-invariant hereditary subalgebras in the sense that if $B\subseteq A$ is $G$-invariant hereditary subalgebra of $A$, then 
$B\rtimes_\mu G$ injects into $A\rtimes_\mu G$.
\item $\rtimes_\mu$ is functorial for ccp maps.
\end{enumerate}
Hence the following corollary is immediate from the fact that $\rtimes_\inj$ is injective.
\end{remark}

 \begin{corollary}\label{cor cor}
 The injective functor $\rtimes_\mu$ is a correspondence functor. \qed
 \end{corollary}
 
 Among the many nice implications of being a correspondence functor, we mention the following which we shall use 
 in Section \ref{sec amenable} below. It follows  directly from Corollary \ref{cor cor} and \cite{Buss:2014aa}*{Theorem 5.6}.

\begin{proposition}\label{prop:inj-duality-functor}
Let $G$ be a locally compact group and let $(A,\alpha)$ be a $G$-algebra. 
Then the crossed product functor $\rtimes_\inj$ is a \emph{duality functor} in the sense that the canonical representation
$$(A,G)\to \M(A\rtimes_\inj G\otimes C^*_\max(G))$$
sending $a\mapsto a\otimes 1$ and $g\mapsto \delta_g\otimes \delta_g$,
extends to an (injective) homomorphism (called the \emph{dual coaction})
$$\dual\alpha_\inj\colon A\rtimes_\inj G\into\M(A\rtimes_\inj G\otimes C^*_\max(G)).$$
\end{proposition}
\medskip
 \noindent
{\bf The $G$-WEP.} We saw in Corollary \ref{cor max} that  for $G$-injective algebras $A$, the injective crossed product by $A$ coincides with the 
maximal crossed product by $A$. We shall now introduce a larger class of $G$-algebras, which enjoy the same property.
Recall that a \cstar{}algebra $A$ has Lance's {\em weak expectation property} (WEP) if every embedding $A\into B$ into another \cstar{}algebra $B$
admits a weak conditional expectation, that is, a ccp map $p\colon B\to A^{**}$ which restricts to the identity on $A$ (see \cite{Brown:2008qy}*{{Definition~3.6.7}}). We  now introduce a $G$-equivariant version of this property:

\begin{definition}
Let $G$ be a locally compact group.
We say that a $G$-algebra $A$ has the {\em $G$-equivariant weak expectation property} ($G$-WEP) if for every $G$-equivariant embedding $\iota\colon A\into B$ into some other $G$-algebra $B$,
there is an equivariant ccp map $p\colon B\to A^{**}$ whose composition with $\iota$ coincides with the canonical inclusion $A\into A^{**}$.
\end{definition}

Here we consider $A^{**}$ endowed with the double dual action $\alpha^{**}\colon G\to \Aut(A^{**})$ of the given action  $\alpha:G\to\Aut(A)$.
Let $A_c^{**}$ denote the subalgebra of $A^{**}$ consisting of all $G$-continuous elements of $A^{**}$. Then for any $G$-algebra $B$ the image 
of any norm decreasing $G$-equivariant map $B\to A^{**}$ lies in $A^{**}_c$. In particular, this applies to the ccp map $p\colon B\to A^{**}$  in the above definition.

\begin{proposition}
Let $G$ be a locally compact group. Then:
\begin{enumerate}
\item Every $G$-injective $G$-algebra has the $G$-WEP.
\item If $A$ is a $G$-algebra such that there exists a $G$-invariant \cstar{}subalgebra $C\sbe A^{**}_c$ which is $G$-injective and contains $A$, then $A$ has the $G$-WEP.
\end{enumerate}
\end{proposition}
\begin{proof} Of course, it suffices to show (2). So assume that $A\sbe C\sbe A^{**}_c$ are as in (2).
Let $\iota\colon A\into B$ be a $G$-equivariant embedding into a $G$-algebra $B$.
Then the $G$-injectivity of $C$ applied to the inclusion $i\colon A\into C$ (which is the co-restriction of the canonical embedding $A\into A^{**}$) implies the existence of a $G$-equivariant ccp map $p\colon B\to C\sbe A^{**}$ with $p\circ \iota =i$.
\end{proof}

\begin{example}
Let $B=\mathcal {B}(H)$ be the \cstar{}algebra of bounded operators on some Hilbert space $H$ endowed with the trivial $G$-action.
We know that the $G$-algebra $\contub(G,B)$ is $G$-injective (with respect to the translation $G$-action). 
Since this is canonically embedded into the double dual of the $G$-algebra $\contz(G,\K)$, where $\K:=\K(H)$, 
it follows from (2) in the above proposition that any $G$-algebra $A$ lying between $\contz(G,\K)$ and $\contub(G,B)$ has the $G$-WEP.
\end{example}

\begin{proposition}\label{prop-GWEP}
If $A$ is a $G$-algebra with the $G$-WEP, then $A\rtimes_\inj G=A\rtimes_\max G$.
\end{proposition}
\begin{proof}
Given a $G$-equivariant embedding $\iota\colon A\into B$, there is a $G$-equivariant ccp map $p\colon B\to A^{**}_{c}$ with $p\circ\iota(a)=a$ for all $a\in A$.
Notice that $A\rtimes_\max G$ embeds into $A^{**}_c\rtimes_\max G$. Indeed, by the universal property of $\rtimes_{\max}$, we have a canonical homomorphism 
$A^{**}_c\rtimes_\max G\to (A\rtimes_\max G)^{**}$ whose compostion
$$A\rtimes_\max G\to A^{**}_c\rtimes_\max G\to (A\rtimes_\max G)^{**}$$
is the canonical bidual embbeding. We can therefore identify $A\rtimes_\max G\sbe A^{**}_c\rtimes_\max G$. 
Now, by functoriality of the maximal crossed product for ccp maps, $p$ induces a ccp map $p\rtimes_\max G\colon B\rtimes_\max G\to A^{**}_c\rtimes_\max G$ satisfying $(p\rtimes_\max G)\circ(\iota\rtimes_\max G)(x)=x$ for all $x\in A\rtimes_\max G$ so that $\iota\rtimes_\max G\colon A\rtimes_\max G\to B\rtimes_\max G$ is injective. Since $B$ was arbitrary the result now follows from the definition of the injective crossed product.
\end{proof}

\begin{remark}
Although above we identified a class of $G$-algebras such that $A\rtimes_\inj G=A\rtimes_\max G$, we shall see later that 
for a locally compact group $G$ the maximal injective functor $\rtimes_\inj$ coincides with $\rtimes_\max$ if and only if 
$G$ is amenable. Indeed, we show in Proposition \ref{group alg} below that the corresponding group algebras coincide
iff $G$ is amenable.
\end{remark}

 \begin{remark}\label{gen rem}
The only property of the maximal crossed product functor used in our constructions for the maximal injective crossed product $\rtimes_\inj$ is its functoriality for $G$-equivariant ccp maps.
  Therefore, the constructions of this section could be carried out without change starting with an arbitrary correspondence functor $\rtimes_\mu$ in place of the maximal crossed product functor $\rtimes_{\max}$.
Everything goes through as before, and the resulting crossed product functor, say $\rtimes_{\inj(\mu)}$ is the largest injective crossed product functor that is dominated by $\rtimes_\mu$. Moreover, for any $G$-injective, algebra $A$ we then have $A\rtimes_{\inj(\mu)}G\cong A\rtimes_\mu G$.
An analoguous statement is not clear for algebras $A$ with the $G$-WEP, since the proof of Proposition \ref{prop-GWEP} uses the 
universality of the maximal crossed product. However, the  proof goes through if we start with an {\em exact} correspondence functor $\rtimes_\mu$ by making use of \cite{Buss:2018pw}*{Theorem 3.5}.
 \end{remark}

\section{Connections with exactness and the LLP}\label{sec LLP}

There are two interesting cases where we can show that the injective crossed product agrees with the reduced crossed product.  Our goal in this section is to discuss these cases, and deduce some consequences: perhaps most notable of these is that we give examples where $\rtimes_r\neq \rtimes_\inj$, and use this to give new examples of groups $G$ for which $C^*_{\max}(G)$ does not have the LLP.

The first such case occurs when $G$ is exact.  We give an ad-hoc definition of exactness that is convenient for our purposes.   See 
\cite{Brodzki:2015kb}*{Theorem A} for a proof that this is equivalent to more standard definitions (the result of \cite{Brodzki:2015kb}*{Theorem A} is only stated for second countable $G$, but the proof works in general with minor modifications). 

\begin{definition}\label{ex def}
A locally compact group $G$ is \emph{exact} if it admits an amenable continuous action on a compact space $X$.
\end{definition}

\begin{proposition}\label{inj ex}
Let $G$ be an exact locally compact group.  Then for any $G$-$C^*$-algebra $A$, $A\rtimes_\inj G=A\rtimes_r G$.  
\end{proposition}

\begin{proof}
As $G$ is exact, $G$ acts continuously and amenably on some compact space $X$.  For any $G$-$C^*$-algebra $A$, we thus have that $(A\otimes C(X))\rtimes_{\max} G=(A\otimes C(X))\rtimes_r G$.  Integrating the covariant representation of $(A,G)$ in $(A\otimes C(X))\rtimes_{\max} G$ given by 
$$
a\mapsto a\otimes 1,\quad g\mapsto \delta_g
$$
gives a $*$-homomorphism 
$$
A\rtimes_{\max} G \to (A\otimes C(X))\rtimes_{\max} G=(A\otimes C(X))\rtimes_{r} G.
$$
As $\rtimes_r$ is injective, this factors through $A\rtimes_r G$.  Thus the reduced norm on $C_c(G,A)$ is one of the norms that $\|\cdot\|_\inj$ is the infimum over, and the result follows.
\end{proof}

For the second example where $\rtimes_\inj=\rtimes_r$, we need to restrict to the case of discrete groups.  We recall an ad-hoc definition of the local lifting property that is convenient for our purposes. See \cite{Brown:2008qy}*{Corollary 13.2.5} for a proof that this is equivalent to the usual definition.

\begin{definition}\label{llp def}
A $C^*$-algebra $A$ has the \emph{local lifting property} (LLP) if for any Hilbert space $H$, $A\otimes \mathcal{B}(H)=A\otimes_{\max}\mathcal{B}(H)$, that is, if there is a unique \cstar{}norm on the algebraic tensor product $A\odot \mathcal{B}(H)$ for every $H$.
\end{definition}

\begin{proposition}\label{inj llp}
Let $G$ be a discrete group such that $C^*_{\max}(G)$ has the local lifting property.  Then for any $G$-$C^*$-algebra $A$, $A\rtimes_\inj G=A\rtimes_r G$.
\end{proposition}

\begin{proof}
Let $\pi:A\to \mathcal{B}(H)$ be any faithful (non-equivariant) $*$-representation, where $\mathcal{B}(H)$ is equipped with the trivial $G$-action.  Let $\widetilde{\pi}:A\to \mathcal{B}(\ell^2(G,H))$ be the amplified form of this representation as in Corollary \ref{g inj exist}, where we equip $\mathcal{B}(\ell^2(G,H))$ with the conjugation action associated to the amplification of the left regular representation $\lambda$.  Then Lemma \ref{realize inj} implies that the integrated form 
$$
\widetilde{\pi}\rtimes G :C_c(G,A)\to \mathcal{B}(\ell^2(G,H))\rtimes_{\max} G
$$
extends to an inclusion 
$$
A\rtimes_{\inj} G \to \mathcal{B}(\ell^2(G,H))\rtimes_{\max} G.
$$
Identify now $\ell^2(G,H)$ with $H\otimes \ell^2(G)$ in the usual way.  As the action of $G$ on $\mathcal{B}(H\otimes \ell^2(G))$ is inner, there is a canonical `untwisting isomorphism'
\begin{align*}
\Phi: \mathcal{B}(H\otimes\ell^2(G))\rtimes_{\max} G & \to  \mathcal{B}(H\otimes \ell^2(G))\otimes_{\max} C^*_{\max}(G) \\ T\delta_g & \mapsto T(1\otimes \lambda_g)\otimes \delta_g.
\end{align*}
On the other hand, using the LLP for $C^*_{\max}(G)$ gives a canonical identification
$$
\mathcal{B}(H\otimes \ell^2(G))\otimes_{\max} C^*_{\max}(G)=\mathcal{B}(H\otimes \ell^2(G))\otimes C^*_{\max}(G),
$$
so we may identify the image of $\Phi$ with the algebra on the right hand side above.  

Consider finally the commutative diagram
$$
\xymatrix{ A\rtimes_{\inj} G \ar[r] \ar[dr]^{\psi} & \mathcal{B}(H\otimes \ell^2(G))\rtimes_{\max} G \ar[d]^-{\Phi} \\ &  \mathcal{B}(H\otimes \ell^2(G))\otimes C^*_{\max}(G)},
$$
where the diagonal arrow $\psi$ is by definition the composition of the other two maps, so in particular injective.  Computing, the diagonal arrow is the integrated form of the covariant pair given on $a\in A$ and $g\in G$ by  
$$
a\mapsto \widetilde{\pi}(a)\otimes 1,\quad g\mapsto 1\otimes \lambda_g\otimes \delta_g.
$$
The image of this map  therefore agrees precisely with the image of $A\rtimes_r G$ under the (injective) composition of the coaction 
$$
\delta:A\rtimes_r G \to A\rtimes_r G \otimes C^*_{\max}(G)
$$
as in \cite{Echterhoff:2006aa}*{Definition A.27} and of the tensor product $*$-homomorphism 
$$
(\widetilde{\pi}\rtimes (1\otimes \lambda))\otimes \text{id}:A\rtimes_r G \otimes C^*_{\max}(G)\to \mathcal{B}(H\otimes \ell^2(G))\otimes C^*_{\max}(G).
$$
As we already remarked that the diagonal arrow $\psi$ is injective, we thus have that the identity map on $C_c(G,A)$ extends to a injection $A\rtimes_\inj G \to A\rtimes_r G$, and are done.
\end{proof}

\begin{corollary}
If $G$ is an exact locally compact group, or if $G$ is discrete and $C^*_{\max}(G)$ has the LLP, then the reduced crossed product is the \emph{only} injective crossed product functor. \end{corollary}
\begin{proof}
If $\rtimes_\mu$ is injective, then Proposition \ref{inj inj} gives that $\rtimes_r\leq \rtimes_\mu\leq \rtimes_\inj$, whence by Propositions~\ref{inj ex} and~\ref{inj llp}, all three are equal.  
\end{proof}

This is in stark contrast to the case of exact crossed products: indeed, if $G$ is any non-amenable group, then there are a large class of exotic exact crossed products arising for example from the Brown-Guentner construction as discussed in \cite{Buss:2015ty}*{Definition 3.6}.

At this point, it is reasonable to ask if $\rtimes_\inj$ \emph{ever} differs from the reduced crossed product!  We can show that this is indeed the case using the relatively explicit construction of non-exact groups due to Osajda \cite{Osajda:2014ys}.  For the proof we need the
 following fact, which is immediate from Lemma \ref{inj exist} and Lemma \ref{realize inj}.

\begin{corollary}\label{l inf max=inj}
For any discrete group $G$, $\ell^\infty(G)\rtimes_{\max} G=\ell^\infty(G)\rtimes_\inj G$. \qed
\end{corollary}

We can now show that $\rtimes_\inj$ is at least sometimes not equal to $\rtimes_r$.  Osajda shows that groups as in the statement exist  \cite{Osajda:2014ys}.

\begin{lemma}\label{not red}
Let $G$ be a non-exact group equipped with an isometric embedding $X\to G$, where $X$ is a coarse union of a sequence of finite connected graphs with a uniform bound on vertex degrees, with girth tending to infinity, and that is an expander.  Then $\ell^\infty(G)\rtimes_\inj G \neq \ell^\infty(G)\rtimes_r G$.
\end{lemma}

\begin{proof}
Using Corollary \ref{l inf max=inj}, it suffices to prove that $\ell^\infty(G)\rtimes_{\max} G \neq \ell^\infty(G)\rtimes_r G$.  Let $\chi_X\in \ell^\infty(G)$ be the characteristic function of $X$.  Then using that $\chi_X(\ell^\infty(G)\rtimes_\alg G)\chi_X$ identifies with the algebraic uniform Roe algebra $\C_u[X]$, it is not too difficult to see that the corners
$$
\chi_X(\ell^\infty(G)\rtimes_{\max} G)\chi_X \quad \text{and}\quad \chi_X(\ell^\infty(G)\rtimes_r G)\chi_X
$$
identify respectively with the maximal and reduced uniform Roe algebras of $X$, denoted $C^*_{u,\max}(X)$ and $C^*_u(X)$.
 Hence it suffices to show that $C^*_{u,\max}(X)$ and $C^*_u(X)$ are not equal.  This can be done $K$-theoretically using the main ideas of \cites{Willett:2010ud,Willett:2010zh}: the basic point is that the maximal coarse Baum-Connes conjecture for $X$ is true, but the usual version is false.  We give a somewhat more direct proof, however, based on \cite{Willett:2013cr}*{Section 8}.  

For this, let $\Delta\in \C_u[X]$ denote the graph Laplacian on $X$; thus if $X=\bigsqcup X_n$ is the decomposition of $X$ into finite connected graphs, we have that $\Delta$ has matrix coefficients given by 
$$
\Delta_{xy}=\left\{\begin{array}{ll} \text{degree}(x) & x=y \\ -1 & x,y \text{ connected by an edge in some $X_n$} \\  0 & \text{ otherwise}\end{array}\right.
$$
According to the definition of $X$ being an expander, there is some $c>0$ such that the spectrum $\text{spec}_{C^*_u(X)}(\Delta)$ of $\Delta$ considered as an element of $C^*_u(X)$ is contained in $\{0\}\cup [c,\infty)$.  On the other hand, \cite{Willett:2013cr}*{Lemma 8.9} combined with the assumption that the girth of the sequence $(X_n)$ tends to infinity implies that the spectrum $\text{spec}_{C^*_{u,\max}(X)}(\Delta)$ of $\Delta$ considered as an element of $C^*_{u,\max}(X)$ contains points in $(0,c]$ for any $c>0$.  Hence $C^*_{u,\max}(X)\neq C^*_u(X)$ as required.
\end{proof}

The following corollary is immediate from Lemma \ref{inj llp} and Lemma \ref{not red}.  

\begin{corollary}\label{not llp}
Let $G$ be as in the hypotheses of Lemma \ref{not red}.  Then $C^*_{\max}(G)$ does not have the LLP. \qed
\end{corollary}

There seem to be very few examples where $C^*_{\max}(G)$ is known not to have the LLP.  We discuss this, and the connection between this property and exactness, in the next few remarks.

\begin{remark}\label{gen llp}
The class of discrete groups $G$ for which $C^*_{\max}(G)$ has the LLP contains all amenable groups, and is closed under taking subgroups, and free products with finite amalgam \cite{Ozawa:2004ab}*{Proposition 3.21 and following discussion}.  However, it is not clear to us that it contains, for example, any non-exact group, or even a group without the Haagerup approximation property.  On the other hand, it appears the only known examples where $C^*_{\max}(G)$ does not have the LLP other than those of Corollary \ref{not llp} are those constructed by Thom in \cite{Thom:2010aa} (other examples where $C^*_{\max}(G)$ does not have the LP were constructed by Ozawa \cite{Ozawa:2004aa}).    
\end{remark}

\begin{remark}\label{llp implies exact}
It is natural to ask whether the LLP for $C^*_{\max}(G)$ implies that $G$ is exact.  Some evidence for this goes as follows.  If $C^*_{\max}(G)$ has the LLP, then Lemma \ref{inj llp} and Corollary \ref{l inf max=inj} imply that $\ell^\infty(G)\rtimes_{\max}G=\ell^\infty(G)\rtimes_r G$.  It would be reasonable (well, arguably...) to expect that this implies that the action of $G$ on the maximal ideal space $\beta G$ of $\ell^\infty(G)$ is amenable, and thus that $G$ is exact.  Note that if $\partial G:=\beta G\setminus G$ is the associated corona of $G$, then the equality $C(\partial G)\rtimes_{\max} G =C(\partial G)\rtimes_r G$ does imply -- indeed characterizes -- that $G$ is exact using the results of \cite{Roe:2013rt}*{Section 5.1}.

On the other hand, if one could produce a non-exact group with $C^*_{\max}(G)$ having the LLP, this would give an example of a non-amenable action on a compact space, such that the associated maximal and reduced crossed products are the same.  This would answer a long-standing open question.
\end{remark}

\begin{remark}\label{exact implies llp}
The converse question, whether exactness of $G$ implies that $C^*_{\max}(G)$ has the LLP, has a negative answer.  Indeed, Thom's example of a group without the LLP from \cite{Thom:2010aa}*{Section 2} is exact.  To construct his example $G$, Thom starts with a specfic (countable) subgroup $G_0$ of $GL_5(R)$, where $R=\mathbb{F}_p[t,t^{-1}]$ is the ring of Laurent polynomials over the finite field with $p$ elements for some prime $p$.  He then defines $G$ to be the quotient of $G_0$ by some specific subgroup $C$ of its center.  Now, $G_0$ is a countable subgroup of $GL_n(R)$ where $R$ is a commutative ring with unit, and therefore has Yu's property A by \cite{Guentner:2013aa}*{Theorem 4.6 and Theorem 5.2.1}.  Hence $C^*_r(G_0)$ is exact by the main result of \cite{Ozawa:2000th}.  On the other hand, as $C$ is a central subgroup of $G_0$, it is abelian, so in particular amenable, and so the quotient map $G_0\to G$ induces a surjective $*$-homomorphism $C^*_r(G_0)\to C^*_r(G)$.  In particular, $C^*_r(G)$ is a quotient of an exact $C^*$-algebra, so exact by 
\cite{Brown:2008qy}*{Corollary 9.4.3}.  Hence $G$ is exact. Similar reasoning shows that the other example of a group not satisfying LLP given in Section~3 of
Thom's paper is exact as well.
\end{remark}

 \section{The injective group algebra, amenability, and the WEP}\label{sec amenable}

We now study the group algebra $C^*_\inj (G):=\C{\rtimes}_\inj G$.  

The first result we are aiming for is a direct analogue of a well-known property for the reduced group $C^*$-algebra of a discrete group
 \cite{Brown:2006aa}*{Corollary 4.1.2}, and provides some evidence that we might have $C^*_\inj(G)=C^*_r(G)$ in general; it does at least show that $C^*_\inj(G)\neq C^*_{\max}(G)$ for a general discrete non-amenable group (and hence, that $\rtimes_\inj\neq \rtimes_\max$ if $G$ is not amenable).

To state the result, we recall one of the definitions of an amenable trace \cite{Brown:2006aa}*{Theorem 3.1.6}.  

\begin{definition}
Let $\tau:A\to \C$ be a tracial state on a unital $C^*$-algebra, let $\pi_\tau:A\to \mathcal{B}(L^2(A,\tau))$ be the associated GNS representation, and let $\pi_\tau(A)''$ be the von Neumann algebra generated by the image of $A$ in this representation.  Then $\tau$ is \emph{amenable} if for any faithful representation $A\subseteq \mathcal{B}(H)$ there is a ucp map $\phi:\mathcal{B}(H)\to \pi_\tau(A)''$ such that $\phi(a)=\pi_\tau(a)$ for all $a\in A$.

We say that a tracial state on a non-unital $C^*$-algebra is \emph{amenable} if its canonical extension to a tracial state on the unitization is amenable.\footnote{We are not sure if there is a standard definition of amenability of a trace on a non-unital $C^*$-algebra; this ad-hoc one is convenient for our purposes.}.
\end{definition}

In other words, the trace $\tau$ is amenable if its GNS representation $\pi_\tau$ is an injective ccp map (in the sense of our Definition~\ref{g inj} for the trivial group) when viewed as a map $A\to \pi_\tau(A)''$. In particular, $\tau$ is amenable if $\pi_\tau(A)''$ is an injective von Neumann algebra (e.g. if $\pi_\tau(A)$ is a nuclear \cstar{}algebra).

\begin{example}\label{at ex}
Let $A$ be a $C^*$-algebra, let $\pi:A\to M_n(\C)$ be a finite-dimensional representation, and let $\text{tr}:M_n(\C)\to \C$ be the canonical tracial state.  Then the pull-back of $\text{tr}$ to (the unitization of) $A$ is amenable.  Indeed, in this case $L^2(A,\tau)$ is finite dimensional by uniqueness of GNS representations, whence $\pi_\tau(A)''$ is finite dimensional, so in particular injective.  The existence of an appropriate $\phi$ thus follows as $\pi_\tau(A)''$ is injective. 
\end{example}

\begin{proposition}\label{group alg}
The group algebra $C^*_{\inj}(G)$ has an amenable trace if and only if $G$ is amenable\footnote{The same property holds for $C^*_r(G)$ in place of $C^*_\inj(G)$, with essentially the same proof; this is well-known, at least when $G$ is discrete \cite{Brown:2006aa}*{Corollary 4.1.2}.}.  In particular, if $G$ is non-amenable then $C^*_{\inj}(G)$ has no finite dimensional representations, and is therefore not equal to $C^*_{\max}(G)$.
\end{proposition}

\begin{proof}
If $G$ is amenable, then $C^*_{\max}(G)=C^*_r(G)$, which forces $C^*_{\max}(G)=C^*_\inj(G)$.  In particular, the trivial representation extends to $C^*_\inj(G)$, and this gives an amenable trace by (a very simple case of) Example \ref{at ex}.  Hence in this case $C^*_\inj(G)$ has an amenable trace.  

Conversely, let $\tau:C^*_\inj(G)\to \C$ be an amenable trace.  Let $A=\widetilde{C^*_\inj(G)}$ be the unitization of $C^*_\inj(G)$ in the non-unital case, or just $A=C^*_\inj(G)$ if this is already unital.  Abuse notation by also writing $\tau:A\to \C$ for the canonical extension.   Fix a non-degenerate embedding $C_{ub}(G)\rtimes_{\max} G\subseteq \mathcal{B}(H)$ and note that Lemmas \ref{inj exist} and \ref{realize inj} give us an embedding 
$$
C^*_\inj(G)\subseteq C_{ub}(G)\rtimes_{\max}G\subseteq \mathcal{B}(H),
$$
and thus also a unital embedding of $A$ into $\mathcal{B}(H)$.  Let $\phi:\mathcal{B}(H)\to \pi_\tau(A)''$ be the ucp map given by the definition of an amenable trace, and let $\tau:\pi_\tau(A)''\to \C$ be the tracial state  induced by $\tau$.  We thus get a state 
$$
\widetilde{m}:\mathcal{B}(H)\to \C, \quad \widetilde{m}:=\tau\circ \phi.
$$ 
We claim that the restriction $m:C_{ub}(G)\to \C$ of $\widetilde{m}$ to $C_{ub}(G)$ is an invariant mean.  Indeed, let $a\in C_{ub}(G)$, write $\alpha$ for the translation action of $G$ on $C_{ub}(G)$, let $g\in G$, and let $(f_i)_{i\in I}$ be an approximate unit in $C_c(G)\subseteq C^*_\inj(G)$.  For each $i$, let $\delta_g* f_i\in C_c(G)$ denote the convolution  of the Dirac mass at $g$ with $f_i$.  Then we have that the net
$$
((\delta_g*f_i)a(\delta_g*f_i)^*)_{i\in I}
$$
converges in the norm of $C_{ub}(G)\rtimes_{\max} G$ to $\alpha_g(a)$.  On the other hand, each $\delta_g*f_i$ is in the multiplicative domain of $\phi$, whence 
$$
m(\alpha_g(a))=\lim_i \tau(\phi((\delta_g*f_i)a(\delta_g*f_i)^*))=\lim_i \tau(\pi_\tau(\delta_g*f_i)\phi(a)\pi_\tau(\delta_g*f_i)^*).
$$
Using that $\tau$ is a trace, this equals $\lim_i \tau(\pi_\tau(f_i^*f_i)\phi(a))$.  As $\pi_\tau:A\to \mathcal{B}(L^2(A,\tau))$ restricts to a nondegenerate representation of $C^*_\inj(G)$, and as $(f_i)$ is an approximate unit for $C^*_\inj(G)$ we have that $\tau(f_i^*f_i)$ converges strongly to the identity operator on $L^2(A,\tau)$; moreover, the canonical extension $\tau:\pi_\tau(A)''\to \C$ is normal, whence in particular strongly continuous on bounded sets.  Thus the net $\lim_i \tau(\pi_\tau(f_i^*f_i)\phi(a))$ converges to $\tau(\phi(a))=m(a)$, completing the proof of invariance of $m$, and thus that $m$ is indeed an invariant mean and $G$ is amenable.

The remaining comments about non-amenable $G$ follow from Example \ref{at ex} and the fact that $C^*_{\max}(G)$ always has at least one finite-dimensional representation (the trivial representation).
\end{proof}

Notice that the amenability condition on a trace $\tau:A\to \C$ has some similarity with the WEP, which we briefly discussed at the end of
Section \ref{basic sec}.
Recall that a \cstar{}algebra $A$ has the WEP if every embedding $A\into B$ 
admits  a ccp map $B\to A^{**}$ which restricts to the identity on $A$. 
By \cite{Brown:2008qy}*{Proposition~3.6.8} this is equivalent to the property that every embedding $A\into B$ induces an embedding $A\otimes_\max D\into B\otimes_\max D$ for every \cstar{}algebra $D$. 
The archetypal example of a \cstar{}algebra with the WEP is the algebra $\mathcal {B}(H)$ of bounded operators on a Hilbert space $H$. On the other hand, the reduced group \cstar{}algebra $C^*_\red(G)$ of a discrete group $G$ has the WEP if and only if $G$ is amenable, see \cite{Brown:2008qy}*{Proposition~3.6.9}. We want to arrive at a similar result for $C^*_\inj(G)$ which gives another hint that $C_\inj^*(G)$  might be equal to 
$C^*_{\red}(G)$. Indeed, we can prove the following general result:

\begin{proposition}\label{prop:inj-WEP}
Let $G$ be a discrete group. If $A$ is a $G$-algebra for which $A\rtimes_\inj G$ has the WEP, then $A\rtimes_\inj G=A\rtimes_\max G$.
\end{proposition}
\begin{proof}
As in the proof of Proposition~\ref{inj llp}, we choose a faithful nondegenerate representation $\pi\colon A\into \mathcal B(H)$ and embed $A\rtimes_\inj G$ into $\mathcal B(H\otimes\ell^2(G))\otimes_\max C^*_\max(G)$ via the diagonal homomorphism $a\delta_g\mapsto \tilde\pi(a)(1\otimes\lambda_g)\otimes\delta_g$. Since $A\rtimes_\inj G$ is assumed to have the WEP, we get an embedding
$$A\rtimes_\inj G\otimes_\max C^*_\max(G)\into\mathcal B(H\otimes\ell^2(G))\otimes_\max C^*_\max(G)\otimes_\max C^*_\max(G).$$
Now we consider the embedding (the comultiplication) 
$$\Delta\colon C^*_\max(G)\into C^*_\max(G)\otimes_\max C^*_\max(G)$$ sending $\delta_g\mapsto \delta_g\otimes\delta_g$.
Using the fact that $\mathcal B(H\otimes\ell^2(G))$ has the WEP, we therefore get an embedding
$$\mathcal B(H\otimes\ell^2(G))\otimes_\max C^*_\max(G)\into \mathcal B(H\otimes\ell^2(G))\otimes_\max C^*_\max(G)\otimes_\max C^*_\max(G).$$
This embedding sends the image of $A\rtimes_\inj G$ in $\mathcal B(H\otimes\ell^2(G))\otimes_\max C^*_\max(G)$ into the image of $A\rtimes_\inj G\otimes_\max C^*_\max(G)$ in $\mathcal B(H\otimes\ell^2(G))\otimes_\max C^*_\max(G)\otimes_\max C^*_\max(G)$. We therefore get an embedding
$$A\rtimes_\inj G\into A\rtimes_\inj G\otimes_\max C^*_\max(G)$$
sending $a\delta_g\mapsto a\delta_g\otimes\delta_g$. It follows from \cite{Buss:2015aa}*{Theorem 5.1} that the dual coaction $\dual\alpha$ from Proposition~\ref{prop:inj-duality-functor} is maximal, which means $A\rtimes_\inj G=A\rtimes_\max G$, as desired.
\end{proof}

\begin{corollary}
For a discrete group $G$, its injective group algebra $C^*_\inj(G):=\C\rtimes_\inj G$ has the WEP if and only if $G$ is amenable.
\end{corollary}
\begin{proof}
This follows directly from Propositions~\ref{group alg} and~\ref{prop:inj-WEP}.
\end{proof}

\begin{remark}
Asking a crossed product to have the WEP is probably a strong restriction. In the above situation it seems to be related to 
the amenability of the underlying action. For example, if $G$ is exact we know that $A\rtimes_\inj G=A\rtimes_\red G$, so the assumption that $A\rtimes_\inj G$ has the WEP implies that $A\rtimes_\max G=A\rtimes_\red G$. If this holds and the crossed product \cstar{}algebra has the WEP, then so does the algebra $A$ as remarked  in \cite{Bhattacharya:2013sj}*{Section~4}.
Moreover, the main result of \cite{Bhattacharya:2013sj} asserts that, assuming the $G$-action on a unital $A$ to be amenable (as defined in \cite{Brown:2008qy}), the crossed product $A\rtimes_\max G=A\rtimes_\red G$ has the WEP if and only if $A$ has the WEP.
\end{remark}

\section{Passing to subgroups}\label{sec subgroups}
Since $\rtimes_\inj$ is a kind of characteristic crossed product functor for a group $G$, it is interesting to see how it behaves 
with respect to passing to subgroups. Recall from \cite{Buss:2015ty}*{Section 6} that given a crossed product functor $\rtimes_\mu$  
its {\em restriction} $\rtimes_{\mu|M}$ to a closed subgroup $M$ is linked to $\rtimes_\mu$ via Green's imprimitivity theorem.
To be more precise, let $(A,\alpha)$ be an $M$-algebra. Then the induced $G$-algebra
$(\Ind_M^GA, \Ind\alpha)$ is defined as
$$\Ind_M^GA:=\left\{F\in C_b(G,A): \begin{matrix} \alpha_h(F(sh))=F(s)\;\forall s\in G, h\in M,\\
\text{and}\; (sM\mapsto \|F(s)\|)\in C_0(G/M)\end{matrix}\right\}$$
and $\big(\Ind\alpha_s(F)\big)(t)=F(s^{-1}t)$ for $F\in \Ind_H^GA$ and $s,t\in G$.
Then Green's imprimitivity theorem provides a canonical $\Ind_M^GA\rtimes_\max G - A\rtimes_\max M$
Morita equivalence $X_M^G(A)$ which is functorial in $A$ (e.g., see \cite{Cuntz:2017nn}*{Chapter 2} for a detailed discussion of 
this theory). Now, given any crossed product functor $\rtimes_\mu$ for $G$, the crossed product $\Ind_M^GA\rtimes_\mu G$ 
is a quotient of $\Ind_M^GA\rtimes_\max G$ by some ideal $I_\mu\subseteq \Ind_M^GA\rtimes_\max G$ which 
corresponds to a unique ideal $J_\mu\subseteq A\rtimes_\max M$ via the Rieffel correspondence such that the 
quotient $X_M^G(A)_\mu:=X_M^G(A)/ (X_M^G(A)\cdot I_\mu)$ becomes an
$$\Ind_M^GA\rtimes_\mu G -A\rtimes_{\mu|M}M:=(A\rtimes_\max M)/J_\mu$$
equivalence bimodule. We show in \cite{Buss:2015ty}*{Section 6}  that $(A,\alpha)\mapsto A\rtimes_{\mu|M}G$ is indeed a crossed product functor 
for $M$ which inherits many important properties from the given functor $\rtimes_\mu$ for $G$.
In what follows next we want to show:

\begin{proposition}\label{prop-restriction}
Let $M$ be a closed subgroup of the locally compact group $G$. Then the restriction $\rtimes_{\inj(G)|M}$ to $M$ of the maximal injective crossed product functor 
$\rtimes_{\inj(G)}$ for $G$ coincides with the maximal injective crossed product $\rtimes_{\inj(M)}$ of $M$.
\end{proposition}

For the proof we need the following lemma, which is a variant of Proposition \ref{inj exist}:

\begin{lemma}\label{lem inj induced}
Suppose that $M$ is a closed subgroup of $G$ and let $(B,\beta)$ be an $M$-injective $M$-algebra. Let 
$$I_M^G(B):=\left\{F\in C_{ub}(G,A): \beta_h(F(sh))=F(s)\;\forall s\in G, h\in M\right\}$$
equipped with $G$-action $\big(I(\beta)_s(F)\big)(t)=F(s^{-1}t)$. Then 
$\big(I_M^G(B), I(\beta)\big)$ is $G$-injective.
\end{lemma}
\begin{proof} The proof is almost identical to the proof of Proposition \ref{inj exist} and is left to the reader.
\end{proof}

\begin{proof}[Proof of Proposition \ref{prop-restriction}]
Let $A$ be any $M$-algebra and let $\varphi:A\into B$ be an $M$-equivariant embedding of $A$ into the $M$-injective 
algebra $B$. By functoriality of  Green's imprimitivity bimodule, we obtain a morphism of imprimitivity bimodules
$\Psi: X_M^G(A)\to X_M^G(B)$ which is compatible with the $*$-homomorphisms $A\rtimes_\max M\to B\rtimes_\max M$ 
and $\Ind_M^GA\rtimes_\max G\to \Ind_M^GB\rtimes_\max G$ induced from the equivariant 
morphisms $\varphi:A\into B$ and its induced form $\Ind\varphi:\Ind_M^GA\into \Ind_M^GB$.
It follows that the kernel $J_\alpha:=\ker\varphi$ is matched to the kernel $I_\alpha:=\ker(\Ind\varphi)$ via 
the Rieffel correspondence with respect to $X_M^G(A)$. Since $B$ is $M$-injective, it follows that the 
quotient $(A\rtimes_\max M)/J_\alpha$ coincides with 
$A\rtimes_{\inj(M)}M$. On the other side we observe that
$\Ind_M^GB$ naturally embeds as a $G$-invariant ideal into $I_M^G(B)$, and therefore 
$\Ind_M^GA$ embeds into the $G$-injective algebra $I_M^G(B)$ via the composition
$$\Ind_M^GA\stackrel{\Ind\varphi}{\into} \Ind_M^GB\into I_M^G(B).$$
Therefore the injective crossed product $\Ind_M^GA\rtimes_{\inj(G)}G$ is equal to the quotient
$(\Ind_M^GA\rtimes_{\max}G)/I_{\inj(G)}$, where $I_{\inj(G)}$ is the kernel of the composition 
$$\Ind_M^GA\rtimes_\max G\to \Ind_M^GB\rtimes_\max G \to I_M^G(B)\rtimes_\max G.$$
But since $\rtimes_\max$ enjoys the ideal property, we see that the second map in this composition is faithful.
Therefore $I$ coincides with the kernel of the first map, which is $I_\alpha$. It follows that 
$A\rtimes_{\inj(M)}M$ is linked to $\Ind_M^GA\rtimes_{\inj(G)}G$ via the Rieffel correspondence
 for $X_M^G(A)$, which proves that $\rtimes_{\inj(M)}=\rtimes_{\inj(G)|M}$.
\end{proof}

\begin{remark}\label{rem G-M-injetive}
Let $M$ be an {\em open} subgroup of the locally compact group $G$ and let $H$ be a Hilbert space.
 Since $\mathcal B(H)$ is an injective $C^*$-algebra, it follows from 
 Proposition \ref{inj exist} that $C_{ub}(G, \mathcal B(H))$ is an injective $G$-algebra. 
 We claim that $C_{ub}(G, \mathcal B(H))$ is also $M$-injective with respect to the restriction of the translation action to $M$.
 To see this we choose a section $\mathfrak s: M\backslash G\to G$ for the space of left $M$-cosets in $G$. 
 Since $M$ is open in $G$, the quotient $M\backslash G$ is a discrete space and we obtain an $M$-equivariant isomorphism
 $$\Psi: C_{ub}(G, \mathcal B(H))\to C_{ub}(M, \ell^\infty(M\backslash G, \mathcal B(H))); \Psi(f)(m,\dot{g})=f(m\cdot \mathfrak s(\dot{g})).$$
 Since $\ell^\infty(M\backslash G, \mathcal B(H))$ is an injective von Neumann algebra (because it is type I), it follows from 
  Proposition \ref{inj exist} that $C_{ub}(M, \ell^\infty(M\backslash G, \mathcal B(H)))$, and hence also $C_{ub}(G, \mathcal B(H))$
  is an injective $M$-algebra.
 \end{remark}

For the maximal and reduced crossed products it is well known that for any open subgroup $M$ of a
locally compact group $G$ and any $G$-algebra $A$, we get an injective embedding of the crossed product
by $M$ into the crossed product by $G$, extending the canonical  inclusion $\iota: C_c(M,A)\to C_c(G,A)$. 
From the above remark we immediately obtain the same property for the injective crossed product:

\begin{proposition}\label{prop subgroup}
Suppose that $M$ is an open subgroup of the locally compact group $G$. Then, if $(A, \alpha)$ is a $G$-algebra, the inclusion 
$\iota: C_c(M,A)\into C_c(G,A)$ extends to a faithful inclusion $A\rtimes_\inj M\into A\rtimes_\inj G$.
\end{proposition}

\begin{proof}
Let $\pi:A\to \B(H)$ be a faithful representation of $A$ on Hilbert space. Let $\tilde\pi:A\to B:=C_{ub}(G, \B(H))$
be the map sending $a$ to the function $[g\mapsto \pi(\alpha_{g^{-1}}(a))]\in B$. Then it follows from the 
above remark together with Lemma \ref{realize inj} that we get the following  commutative diagram of maps
$$
\xymatrix{ A\rtimes_\inj G \ar@{^(->}[r] &B\rtimes_\max G\\
A\rtimes_\inj M \ar@{^(-->}[u]  \ar@{^(->}[r] &B\rtimes_\max M \ar@{^(->}[u]}
$$
where the broken arrow exists and extends the inclusion $\iota:C_c(M,A)\to C_c(G,A)$ 
because of injectivity of all other maps in the diagram 
and commutativity on the level of $C_c(M,A)$.
\end{proof}

\section{Questions}\label{sec questions}

\begin{enumerate}
\item What is $\C\rtimes_\inj G$?  The only information we currently have comes from Proposition \ref{group alg} in general, plus Proposition \ref{inj ex} and Proposition \ref{inj llp} in some special cases.  All of these results provide some evidence that $C^*_\inj(G)$ might be equal to $C^*_r(G)$ in general, but we have no strong feeling about this.

For a discrete group $G$, using the representation from the proof of Proposition~\ref{inj llp}, notice that $C^*_\inj (G)$ identifies with the $C^*$-algebra generated by the ``diagonal'' representation
$$G\to \B(\ell^2(G))\otimes_{\max} C^*_{\max}(G),\quad g\mapsto \lambda_g\otimes\delta_g.$$
It follows that $C^*_\inj(G)=C^*_r (G)$ iff this representation factors through $C^*_r(G)$. Is this always true? We know that it is true if $G$ is exact or $C^*_{\max}(G)$ has the LLP.  Similarly, we have that for any locally compact $G$, $C^*_{\inj}(G)$ agrees with the image of the natural map
$$
C^*_{\max}(G)\to C_{ub}(G)\rtimes_{\max}G
$$
induced by the unit inclusion $\C\to C_{ub}(G)$, and one can ask if this map always factors through the reduced group $C^*$-algebra.

\item Does the LLP for $C^*_{\max}(G)$ imply exactness of $G$?  Evidence for a positive answer is provided by Remark \ref{llp implies exact}, and the fact that Corollary \ref{not llp} shows that the `best understood' examples of non-exact groups are such that $C^*_{\max}(G)$ does not have the LLP.  Note that the converse is false by Remark \ref{exact implies llp}.

\item More generally, is $\rtimes_\inj$ always different from $\rtimes_r$ for non-exact groups?  This would be implied by $C_{ub}(G)\rtimes_{\max}G\neq C_{ub}(G)\rtimes_r G$ for all non-exact groups, which matches the (scant) available evidence.

\item Is $\rtimes_\inj$ exact?  More generally, can a non-exact group admit a crossed product functor that is both exact and injective?
It would also be interesting to compare the injective crossed product functor $\rtimes_\inj$ with the minimal exact crossed product functor 
$\rtimes_{\mathcal E}$ of \cite{Buss:2018pw}. Both functors agree for exact groups with the reduced crossed product functor and 
so far we do not know of any example of a group $G$ for which $\rtimes_{\inj}\neq \rtimes_{\mathcal E}$. 

\item 
Is $\rtimes_\inj$ a KLQ-functor? This is related to  Proposition \ref{prop:inj-duality-functor}. More precisely, it is equivalent to the existence of a faithful homomorphism 
$$A\rtimes_\inj G\into \M(A\rtimes_\max G\otimes C^*_\inj(G))$$ 
extending the representation $a\delta_g\mapsto a\delta_g\otimes \delta_g$.
Notice that if $C^*_\inj(G)=C^*_\red(G)$, then $\rtimes_\inj$ can only be a KLQ-functor if it equals $\rtimes_\red$. 
\end{enumerate}

\bibliographystyle{abbrv}


\begin{bibdiv}
\begin{biblist}

\bib{Baum:2013kx}{article}{
      author={Baum, Paul},
      author={Guentner, Erik},
      author={Willett, Rufus},
       title={Expanders, exact crossed products, and the {B}aum-{C}onnes
  conjecture},
        date={2015},
     journal={Ann. ${K}$-theory},
      volume={1},
      number={2},
       pages={155\ndash 208},
}

\bib{Bhattacharya:2013sj}{article}{
      author={Bhattacharya, Angshuman},
      author={Farenick, Douglas},
       title={Crossed products of ${C^*}$-algebras with the weak expectation
  property},
        date={2013},
     journal={New York J. Math.},
      volume={19},
       pages={423\ndash 429},
}

\bib{Brodzki:2015kb}{article}{
      author={Brodzki, Jacek},
      author={Cave, Christopher},
      author={Li, Kang},
       title={Exactness of locally compact second countable groups},
        date={2017},
     journal={Adv. Math.},
      volume={312},
       pages={209\ndash 233},
}

\bib{Brown:2006aa}{article}{
      author={Brown, Nathanial},
       title={Invariant means and finite representation theory of
  ${C^*}$-algebras},
        date={2006},
     journal={Mem. Amer. Math. Soc.},
      volume={184},
      number={865},
       pages={vi+105pp},
}

\bib{Brown:2011fk}{article}{
      author={Brown, Nathanial},
      author={Guentner, Erik},
       title={New ${C}^*$-completions of discrete groups and related spaces},
        date={2013},
     journal={Bull. London Math. Soc.},
      volume={45},
      number={6},
       pages={1181\ndash 1193},
}

\bib{Brown:2008qy}{book}{
      author={Brown, Nathanial},
      author={Ozawa, Narutaka},
       title={${C}^*$-algebras and finite-dimensional approximations},
      series={Graduate Studies in Mathematics},
   publisher={American Mathematical Society},
        date={2008},
      volume={88},
}

\bib{Buss:2015aa}{article}{
      author={Buss, Alcides},
      author={Echterhoff, Siegfried},
       title={Maximality of dual coactions on sectional ${C^*}$-algebras of
  {F}ell bunles and applications},
        date={2015},
     journal={Studia Math.},
      volume={229},
      number={3},
       pages={233\ndash 262},
}

\bib{Buss:2015ty}{inproceedings}{
      author={Buss, Alcides},
      author={Echterhoff, Siegfried},
      author={Willett, Rufus},
       title={Exotic crossed products},
        date={2015},
   booktitle={Operator algebras and applications},
      editor={Carlsen, Toke~M.},
      editor={Larsen, Nadia~S.},
      editor={Neshveyev, Sergey},
      editor={Skau, Christian},
      series={The Abel Symposium},
   publisher={Springer},
       pages={61\ndash 108},
}

\bib{Buss:2014aa}{article}{
      author={Buss, Alcides},
      author={Echterhoff, Siegfried},
      author={Willett, Rufus},
       title={Exotic crossed products and the {B}aum--{C}onnes conjecture},
        date={2018},
        ISSN={0075-4102},
     journal={J. Reine Angew. Math.},
      volume={740},
       pages={111\ndash 159},
         url={https://doi.org/10.1515/crelle-2015-0061},
      review={\MR{3824785}},
}

\bib{Buss:2018pw}{unpublished}{
      author={Buss, Alcides},
      author={Echterhoff, Siegfried},
      author={Willett, Rufus},
       title={The minimal exact crossed product},
        date={2018},
        note={arXiv:1804.02725},
}

\bib{Cuntz:2017nn}{book}{
      author={Cuntz, Joachim},
      author={Echterhoff, Siegfried},
      author={Li, Xin},
      author={Yu, Guoliang},
       title={${K}$-theory for group ${C^*}$-algebras and semigroup
  ${C^*}$-algebras},
      series={Oberwolfach Seminars},
   publisher={Birkh\"{a}user},
        date={2017},
}

\bib{Echterhoff:2006aa}{article}{
      author={Echterhoff, Siegfried},
      author={Kaliszewski, S.},
      author={Quigg, John},
      author={Raeburn, Iain},
       title={A categorical approach to imprimitivity theorems for
  ${C^*}$-dynamical systems},
        date={2006},
     journal={Mem. Amer. Math. Soc.},
      volume={180},
}

\bib{Guentner:2013aa}{article}{
      author={Guentner, Erik},
      author={Tessera, Romain},
      author={Yu, Guoliang},
       title={Discrete groups with finite decomposition complexity},
        date={2013},
     journal={Groups, Geometry and Dynamics},
      volume={7},
      number={2},
       pages={377\ndash 402},
}

\bib{Hamana:1985aa}{article}{
      author={Hamana, Masamichi},
       title={Injective envelopes of ${C^*}$-dynamical systems},
        date={1985},
     journal={Tohoku Math. J.},
      volume={37},
       pages={463\ndash 487},
}

\bib{Kaliszewski:2012fr}{article}{
      author={Kaliszewski, S.},
      author={Landstad, Magnus},
      author={Quigg, John},
       title={Exotic group ${C}^*$-algebras in noncommutative duality},
        date={2013},
     journal={New York J. Math.},
      volume={19},
       pages={689\ndash 711},
}

\bib{Osajda:2014ys}{unpublished}{
      author={Osajda, Damian},
       title={Small cancellation labellings of some infinite graphs and
  applications},
        date={2014},
        note={arXiv:1406.5015},
}

\bib{Ozawa:2000th}{article}{
      author={Ozawa, Narutaka},
       title={Amenable actions and exactness for discrete groups},
        date={2000},
     journal={C. R. Acad. Sci. Paris S{\'e}r. I Math.},
      volume={330},
       pages={691\ndash 695},
}

\bib{Ozawa:2004ab}{article}{
      author={Ozawa, Narutaka},
       title={About the {QWEP} conjecture},
        date={2004},
     journal={Internat. J. Math.},
      volume={15},
       pages={501\ndash 530},
}

\bib{Ozawa:2004aa}{article}{
      author={Ozawa, Narutaka},
       title={There is no separable universal {$II_1$}-factor},
        date={2004},
     journal={Proc. Amer. Math. Soc.},
      volume={132},
      number={2},
       pages={487\ndash 490},
}

\bib{Roe:2013rt}{article}{
      author={Roe, John},
      author={Willett, Rufus},
       title={Ghostbusting and property {A}},
        date={2014},
     journal={J. Funct. Anal.},
      volume={266},
      number={3},
       pages={1674\ndash 1684},
}

\bib{Ruan:2016kl}{article}{
      author={Ruan, Zhong-Jin},
      author={Wiersma, Matthew},
       title={On exotic group ${C^*}$-algebras},
        date={2016},
     journal={J. Funct. Anal.},
      volume={271},
      number={2},
       pages={437\ndash 453},
}

\bib{Thom:2010aa}{article}{
      author={Thom, Andreas},
       title={Examples of hyperlinear groups without the factorization
  property},
        date={2010},
     journal={Groups, Geom. Dynam.},
      volume={4},
       pages={195\ndash 208},
}

\bib{Wiersma:2015tx}{article}{
      author={Wiersma, Matthew},
       title={${L}^p$-{F}ourier and {F}ourier-{S}tieltjes algebras for locally
  compact groups},
        date={2015},
     journal={J. Funct. Anal.},
      volume={269},
      number={12},
       pages={3928\ndash 3951},
}

\bib{Wiersma:2016kr}{article}{
      author={Wiersma, Matthew},
       title={Constructions of exotic group ${C^*}$-algebras},
        date={2016},
     journal={Illinois J. Math},
      volume={60},
      number={3-4},
       pages={655\ndash 667},
}

\bib{Willett:2010ud}{article}{
      author={Willett, Rufus},
      author={Yu, Guoliang},
       title={Higher index theory for certain expanders and {G}romov monster
  groups {I}},
        date={2012},
     journal={Adv. Math.},
      volume={229},
      number={3},
       pages={1380\ndash 1416},
}

\bib{Willett:2010zh}{article}{
      author={Willett, Rufus},
      author={Yu, Guoliang},
       title={Higher index theory for certain expanders and {G}romov monster
  groups {II}},
        date={2012},
     journal={Adv. Math.},
      volume={229},
      number={3},
       pages={1762\ndash 1803},
}

\bib{Willett:2013cr}{article}{
      author={Willett, Rufus},
      author={Yu, Guoliang},
       title={Geometric property ({T})},
        date={2014},
     journal={Chinese Ann. Math. Ser. B},
      volume={35},
      number={5},
       pages={761\ndash 800},
}

\end{biblist}
\end{bibdiv}

\end{document}